\documentclass[11 pt]{amsart}
\usepackage{amsmath, amsthm,amssymb,bbm,enumerate,mathrsfs,mathtools}
\usepackage{a4wide}
\usepackage{mdwtab}
\usepackage{tikz,xcolor,verbatim}
\usetikzlibrary{calc,shapes}
\usepackage{hyperref}
\usepackage{makecell}
\usepackage{mathrsfs}
\usepackage[numbers,sort&compress]{natbib}
\usepackage{cancel}
\usepackage{ulem}
\usepackage{lineno}
\usepackage{complexity}
\usepackage[defaultlines=3,all]{nowidow}
\usepackage{microtype}
\usetikzlibrary{shapes.geometric}

\tikzset{
    triangle/.style={
        draw,
        shape border rotate=0,
        regular polygon,
        regular polygon sides=3,
        fill=white,
        node distance=2cm,
        minimum height=4em
    }
}
\usepackage{amssymb}
\usepackage{dirtytalk}
\usepackage{algorithmicx}
\usepackage{algpseudocode}
\usepackage{algorithm}
\usepackage{a4wide}
\algblock{Input}{EndInput}
\algnotext{EndInput}
\algblock{Output}{EndOutput}
\algnotext{EndOutput}
\algblock{Initialization}{EndInitialization}
\algnotext{EndInitialization}
\algblock{Begin}{EndBegin}
\algnotext{EndBegin}

 \usepackage[yyyymmdd,hhmmss]{datetime}
 \usepackage{background}
 \SetBgContents{ver. \currenttime \qquad \today }
 \SetBgScale{1}
 \SetBgAngle{0}
 \SetBgOpacity{1}
 \SetBgPosition{current page.south west} 
 \SetBgHshift{3cm}
 \SetBgVshift{0.5cm} 

\tikzset{
  bigblue/.style={circle, draw=blue!80,fill=blue!40,thick, inner sep=1.5pt, minimum size=5mm},
  bigred/.style={circle, draw=red!80,fill=red!40,thick, inner sep=1.5pt, minimum size=5mm},
  bigblack/.style={circle, draw=black!100,fill=black!40,thick, inner sep=1.5pt, minimum size=5mm},
  bluevertex/.style={circle, draw=blue!100,fill=blue!100,thick, inner sep=0pt, minimum size=2mm},
  redvertex/.style={circle, draw=red!100,fill=red!100,thick, inner sep=0pt, minimum size=2mm},
  blackvertex/.style={circle, draw=black!100,fill=black!100,thick, inner sep=0pt, minimum size=1.5mm},  
  whitevertex/.style={circle, draw=black!100,fill=white!100,thick, inner sep=0pt, minimum size=2mm},  
  smallblack/.style={circle, draw=black!100,fill=black!100,thick, inner sep=0pt, minimum size=1mm},
  smallwhite/.style={circle, draw=white!100,fill=white!100,thick, inner sep=0pt, minimum size=1mm},
  redvertexv2/.style={circle, draw=black!100,fill=red!100,thick, inner sep=0pt, minimum size=3.35mm},
  bluevertexv2/.style={circle, draw=black!100,fill=blue!100,thick, inner sep=0pt, minimum size=3.35mm},
  greenvertexv2/.style={circle, draw=black!100,fill=green!100,thick, inner sep=0pt, minimum size=3.35mm},
  blackvertexv2/.style={circle, draw=black!100,fill=black!100,thick, inner sep=0pt, minimum size= 3.35mm}
}

\makeatletter
\pgfdeclareshape{myNode}{
  \inheritsavedanchors[from=rectangle] 
  \inheritanchorborder[from=rectangle]
  \inheritanchor[from=rectangle]{center}
  \inheritanchor[from=rectangle]{north}
  \inheritanchor[from=rectangle]{south}
  \inheritanchor[from=rectangle]{west}
  \inheritanchor[from=rectangle]{east}
  \backgroundpath{
    \southwest \pgf@xa=\pgf@x \pgf@ya=\pgf@y
    \northeast \pgf@xb=\pgf@x \pgf@yb=\pgf@y
    \pgfsetcornersarced{\pgfpoint{5pt}{5pt}}
    \pgfpathmoveto{\pgfpoint{\pgf@xa}{\pgf@ya}}
    \pgfpathlineto{\pgfpoint{\pgf@xa}{\pgf@yb}}
    \pgfpathlineto{\pgfpoint{\pgf@xb}{\pgf@yb}}
    \pgfsetcornersarced{\pgfpoint{5pt}{5pt}}
    \pgfpathlineto{\pgfpoint{\pgf@xb}{\pgf@ya}}
    \pgfpathclose
 }
}
\makeatother

\usepackage{xcolor}
\newcommand\ben[1]{\textcolor{magenta}{#1}}

\newcommand\dan[1]{\textcolor{green}{#1}}
\newcommand\pom[1]{\textcolor{teal}{#1}}

\newenvironment{claimproof}{ \trivlist
	\item[\hskip\labelsep
	\textit{Proof of the claim}.]\ignorespaces
}{\hfill$\vartriangleleft$\medskip
	
}

\title[Subchromatic numbers of powers of graphs with excluded minors]{Subchromatic numbers of powers of graphs with excluded minors}

\author[Cortés]{Pedro P. Cort\'es}
\address[Pedro P. Cort\'es]{Departamento de Ingenier\'ia Matem\'atica, Universidad de Chile, Santiago, Chile}
\email{pedrocortes1997@gmail.com}
\author[Kumar]{Pankaj Kumar}
\address[Pankaj Kumar]{Department of Computer Science, Heinrich Heine Universit\"{a}t, D\"{u}sseldorf , Germany}
\email{pundir.pankaj25@gmail.com}
\author[Moore]{Benjamin Moore}
\address[Benjamin Moore]{Institute of Computer Science, Klosterneuburg Austria}  
\email{Benjamin.Moore@ist.ac.at}

\author[Ossona de Mendez]{Patrice Ossona de Mendez}
\address[Patrice Ossona de Mendez]{Centre d’Analyse et de Math\'ematiques Sociales (UMR 8557) and CNRS, Paris, France}
\email{pom@ehess.fr}
\author[Quiroz]{Daniel A. Quiroz}
\address[Daniel A. Quiroz]{Instituto de Ingenier\'ia Matem\'atica -- CIMFAV, Universidad de Valpara\'iso, Valpara\'iso, Chile}
\email{daniel.quiroz@uv.cl}

\thanks{P.P. Cortés and D.A. Quiroz are partially supported by ANID/Fondecyt Iniciación en Investigación 11201251, and by MATHAMSUD MATH210008.}
\thanks{B. Moore completed this project while he was a postdoctoral research at Charles University. He was supported by project 22-17398S (Flows and cycles in graphs on surfaces) of Czech Science Foundation when this project was completed.}
\thanks{Pankaj Kumar partially completed this work while a Ph.D student at Charles university.}

\date{}

\newtheorem{thm}[equation]{Theorem}
\newtheorem{lemma}[equation]{Lemma}

\newtheorem{claim}{Claim}
\newtheorem{question}[equation]{Question}

\theoremstyle{definition}
\newtheorem{definition}[equation]{Definition}

\newtheorem*{ack}{Acknowledgements}

\newtheoremstyle{case}{}{}{\normalfont}{}{\itshape}{\normalfont:}{ }{}

\theoremstyle{case}

\numberwithin{equation}{section}

\newcommand{\N}{\mathbb{N}}
\DeclareMathOperator{\subcol}{swcol}
\DeclareMathOperator{\col}{col}
\DeclareMathOperator{\gcol}{gcol}
\DeclareMathOperator{\wcol}{wcol}
\DeclareMathOperator{\WReach}{WReach}
\DeclareMathOperator{\GReach}{GReach}
\DeclareMathOperator{\SubReach}{SemiReach}
\DeclareMathOperator{\Reach}{Reach}

\DeclareMathOperator{\semiweak}{swcol}
\newcommand{\chisub}{\chi_{\text{\rm sub}}}
\newcommand{\floor}[1]{\left\lfloor#1\right\rfloor }
\newcommand{\ceil}[1]{\left\lceil#1\right\rceil }

\date{}

\begin{document}
\maketitle

\begin{abstract}
    A \textit{$k$-subcolouring} of a graph $G$ is a function $f:V(G) \to \{0,\ldots,k-1\}$ such that the set of vertices coloured $i$ induce a disjoint union of cliques. The \textit{subchromatic number}, $\chisub(G)$, is the minimum $k$ such that $G$ admits a $k$-subcolouring. Ne\v{s}et\v{r}il, Ossona de Mendez, Pilipczuk, and Zhu~(2020), recently  raised the problem of finding tight upper bounds for $\chisub(G^2)$ when $G$ is planar. We show that $\chisub(G^2)\le 43$ when $G$ is planar, improving their bound of 135. We  give even better bounds when the planar graph $G$ has larger girth. Moreover, we show that $\chisub(G^{3})\le 95$, improving the previous bound of 364. For these we adapt some recent techniques of Almulhim and Kierstead (2022), while also extending the decompositions of triangulated planar graphs of Van den Heuvel, Ossona de Mendez, Quiroz, Rabinovich and Siebertz (2017), to planar graphs of arbitrary girth. Note that these decompositions are the precursors of the graph product structure theorem of planar graphs.
    
    We give  improved bounds for $\chisub(G^p)$ for all $p$, whenever $G$ has bounded treewidth, bounded simple treewidth, bounded genus, or excludes a clique or biclique as a minor. For this we introduce a family of parameters which form a gradation between the strong and the weak colouring numbers. We give upper bounds for these parameters for graphs coming from such  classes. 

    Finally, we give a 2-approximation algorithm for the subchromatic number of graphs 
    having a layering in which each layer has bounded cliquewidth and this layering is computable in polynomial time (like the class of all $d^{th}$ powers of planar graphs, for fixed $d$).
    This algorithm works even if the power $p$ and the graph $G$ is unknown.
\end{abstract}

\section{Introduction}

 In this paper, we study a notion which allows us to ``colour" dense graphs with few colours. Recall that a \emph{$k$-colouring} is a function $f:V(G) \to \{0,\ldots,k-1\}$ such that for all $e =xy \in E(G)$, we have $f(x) \neq f(y)$.  A \emph{$k$-subcolouring} of a graph $G$ is a function $f:V(G) \to \{0,\ldots,k-1\}$ such that the set of vertices coloured $i$ induce a disjoint union of cliques  (first defined in \cite{ALBERTSON198933}). We let $\chisub(G)$ be the \emph{subchromatic number} of $G$, that is, the minimum $k$ such that $G$ admits a $k$-subcolouring. Of course, if $G$ is $k$-colourable, then $G$ is $k$-subcolourable. However the converse is far from the truth, as cliques are $1$-subcolourable, and hence subcolouring gives us a notion for colouring dense graphs with few colours. 

 In general, deciding if a graph is $k$-subcolourable is NP-complete \cite{ACHLIOPTAS199721}. In particular, it is NP-complete to determine if a triangle-free planar graph has subcolouring number at most~$2$~\cite{GIMBEL2003139}. 
 Also, for any $\varepsilon>0$, the subchromatic number of $n$-vertex graphs cannot be approximated within a factor of $n^{1/2-\varepsilon}$ in polynomial time, unless $\NP\subseteq\ZPP$\footnote{{\ZPP}  (zero-error probabilistic polynomial time) is the complexity class of problems for which a probabilistic Turing machine exists, which
 always returns the correct answer within a running time polynomial in expectation for every input.}, see \cite{noapprox}.
 On the other hand a recent result of Ne\v{s}et\v{r}il, Ossona de Mendez, Pilipczuk and Zhu~\cite{clusteringpowers}, gives constant upper bounds for the subchromatic number of powers of graphs coming from sparse classes. To formalize this statement we need some definitions.

First recall that for a graph $G$, the \emph{$d^{\rm\,th}$ power} of $G$ is the graph $G^{d}$ with vertex set $V(G)$ and $uv \in E(G^{d})$ if there is $u,v$-path of length at most $d$ in $G$. Note that the $d^{\rm\,th}$-power of a graph is generally very dense, even if the original graph is sparse. For example, the square of a star is a clique. 

Now we introduce a gradation between the weak and strong colouring numbers of  Kierstead and Yang \cite{Weakcolouring}, which is useful for this paper. We use the term \emph{generalised colouring numbers} to refer to this family of parameters. 

For $k,\ell \in \mathbb{N}\cup\{\infty\}$, a graph $G$, a linear ordering $\sigma$ of $V(G)$, and two vertices $u,v$ satisfying $u \leq_\sigma v$, we say that $u$ is \emph{$k$-hop $\ell$-reachable} from $v$ if there exists a  path $P=x_0x_1...x_s$ with $x_0=v$, $x_s=u$ and $s\leq\ell$, such that  $u<_{\sigma}x_{i-1}$ for every $i\in[s]$, and such that $|\{j\in[s] \mid x_{j}<_{\sigma}x_{i-1}\text{  for every } i\in[j]\}|\leq k$. The set of vertices that are $k$-hop $\ell$-reachable from $v$ with respect to $\sigma$ is denoted by $\GReach_{k,\ell}[G,\sigma,v]$. Note that in particular $v \in \GReach_{k,\ell}[G,\sigma,v]$. The \emph{$k$-hop $\ell$-colouring number of $\sigma$} is  $\gcol_{k,\ell}(G,\sigma) =  \max \{|\GReach_{k,\ell}[G,\sigma,v]| \mid v \in V(G) \}$ and the \emph{$k$-hop $\ell$-colouring number of $G$}, denoted $\gcol_{k,\ell}(G)$, is the smallest $\gcol_{k,\ell}(G,\sigma)$ for $\sigma$ ranging over all possible linear orderings of $V(G)$. The parameters $\col_\ell(G): =\gcol_{1,\ell}(G)$ and $\wcol_\ell(G):=\gcol_{\ell,\ell}(G)$ are  \emph{the strong and the weak colouring numbers}, respectively. Note for instance that we have $$\gcol_{k+k',\ell+\ell'}(G)\le \gcol_{k,\ell}(G)\cdot \gcol_{k',\ell'}(G),$$ 
as paths coming from $\gcol_{k+k',\ell + \ell'}(G)$ can be split into two smaller paths.  In particular $$\gcol_{k,\ell}(G)\le \col_{\ell-k+1}(G)\cdot \col_1(G)^{k-1} \le \wcol_{\ell-k+1}(G)\cdot \wcol_1(G)^{k-1}.$$

A class $\mathscr C$ of graphs  has \emph{bounded expansion} if, for every $\ell \in \N$, $\sup_{G\in\mathscr C}\text{wcol}_{\ell}(G) < \infty$  \cite{weakcolouringzhu}. Many natural graph classes have bounded expansion; for example, planar graphs, any graph class omitting a graph $H$ as a minor, or even more generally, any graph class omitting $H$ as a topological minor. We refer the reader to \cite{sparsity} for a more detailed treatment of bounded expansion classes.

With this in hand, we can state the bound obtained in
 \cite{clusteringpowers} for the subchromatic number of powers of graphs.

\begin{thm}[Ne\v{s}et\v{r}il et al. \cite{clusteringpowers}]
\label{NOdMPZ}
For any graph $G$, and any fixed integer $d \in \mathbb{N}$, we have $\chisub(G^d) \leq \wcol_{2d}(G)$. 
\end{thm}

In particular, if $\mathcal{C}$ is a class of bounded expansion, then for any fixed integer $d$, there exists an integer $c=c(\mathcal{C},d)$ such that $\chisub(G^{d}) \leq c$, for any $G \in \mathcal{C}$. 

Our first result is a refinement of this upper bound as follows.

\begin{thm}
\label{thm:gcol}
For any graph $G$, and any fixed integer $d \in \N$, we have $\chisub(G^d) \leq \gcol_{d,2d}(G)$. Moreover if $d$ is odd, then we have $\chisub(G^d) \leq \gcol_{d,2d-1}(G)$.
\end{thm}
For the purpose of our main result, we prove a stronger version of this theorem (see Theorem~\ref{semiweak}). Toward this end, we follow a similar template to that used to prove Theorem \ref{NOdMPZ} in \cite{clusteringpowers}, while using some ideas from the study of the chromatic numbers of exact distance graphs of Van~den~Heuvel, Kierstead, and Quiroz~\cite{vandenHeuveletal2019}.

Theorem \ref{thm:gcol} motivates us to give bounds for the generalised colouring numbers in various minor closed classes. These results are interesting in their own right and likely have many applications. In particular, they imply tighter bounds on the subchromatic number of arbitrary powers of graphs in these classes. Our results on the generalised colouring numbers are summed up in Table~\ref{tab:gcol}. There, $K^{*}_{s,t}$ refers to the complete join $K_s+\overline{K_t}$, 
that is the graph that can be partitioned into a clique of size $s$ and an independent set of size $t$ in such a a way that every vertex in the independent set is adjacent to every vertex in the clique\footnote{Such graphs are sometimes called complete split graphs.}. It is useful to compare these to the  known bounds on the weak colouring numbers given in Table~\ref{tab:weak}. 


\begin{table}[h]
\begin{center}
\begin{tabular}{|c|c|}
\hlx{hvv}
\textbf{Constraint on $G$}&\textbf{Upper bound for $\gcol_{k,\ell}(G)$}\\
\hlx{vvhhvv}
treewidth at most $t$&$\gcol_{k,\ell}(G)\leq\binom{t+k}{t}$\\
\hlx{vvhvv}
simple treewidth at most $t$&$\gcol_{k,\ell}(G)\leq (k+1)^{t-1}(\lceil \log k \rceil +2\lfloor\ell/k\rfloor)$\\
\hlx{vvhvv}
genus at most $g$&$\gcol_{k,\ell}(G)\leq \bigl(2g+\binom{k+2}{2}-1\bigr)(2\ell+1)+\ell+1$\\
\hlx{vvhvv}
$K_t$-minor free, $t\geq4$&$\gcol_{k,\ell}(G)\leq \bigl(\binom{t+k-2}{t-2}-1\bigr)(t-3)(2\ell+1)+(t-3)\ell+1$\\
\hlx{vvhvv}
$K^*_{2,t}$-minor free, $t\geq 2$& $\gcol_{k,\ell}(G)\leq (t-1)(k(2\ell+1)+\ell)+1$\\
\hlx{vvhvv}
$K^*_{3,t}$-minor free& 
$\gcol_{k,\ell}(G)\leq \bigl(\binom{k+2}{2}-1\bigr)(2t+1)(2\ell+1)+(2t+1)\ell+1$\\
\hlx{vvhvv}
$K^*_{s,t}$-minor free, $t\geq 2$& 
$\gcol_{k,\ell}(G)\leq s(t-1)\binom{s+k}{s}(2\ell+1)-\ell$\\
\hlx{vvh}
\end{tabular}
\end{center}
\caption{Upper bounds on $\gcol_{k,\ell}(G)$ according to several constraints on $G$.}
\label{tab:gcol}
\end{table}

We highlight that for graphs with bounded treewidth we obtain no dependency on $\ell$, and that $\chisub(G^2)\le 6$  when $G$ has treewidth~2.   For graphs with bounded genus and those with excluded minors, we obtain no dependency on $\ell$ for the binomial coefficients of the corresponding bounds, thus these bounds have a linear dependency on~$\ell$ which is known to be best possible even for planar graphs (see e.g. \cite[Proposition A.2]{NPW22}).

Our main result gives even tighter bounds on the subchromatic number of squares and cubes of planar graphs. Previously, the best bound was derived from Theorem \ref{NOdMPZ} and the following theorem due to Van den Heuvel, Ossona de Mendez, Quiroz, Rabinovich, and Siebertz \cite{vandenHeuveletal2017}. 
\begin{thm}[Van den Heuvel et al. \cite{vandenHeuveletal2017}]\label{thm:wcolplanar}
If $G$ is planar we have $\wcol_{d}(G) \leq \binom{d+2}{2}(2d+1)$.
\end{thm}

Thus,  if $G$ is a planar graph, then $\chisub(G^{2})$ is at most $135$, and $\chisub(G^{3})$ is at most $364$. Improving the bounds in the case of squares of planar graphs was proposed as a problem in~\cite{clusteringpowers}. This problem is very natural considering, for instance, the attention that (usual) colouring of squares of planar graphs has deserved (see e.g. \cite{Aminietal,Cranston,Thomassen} and the references therein). The main result of this paper is the following improvement of that bound. (Recall that the \emph{girth} of a graph is infinity if the graph is acyclic, or the length of its shortest cycle otherwise.)

\begin{thm}
\label{thm:girthbounds}
For any planar graph $G$ with girth $g$, we have 
\begin{equation*}
\chisub(G^{2}) \leq
    \begin{cases}
        43 & \text{if } g \geq 3 \\
        39 & \text{if } g \geq 10 \\
        15 & \text{if } g \geq 17 \\
    \end{cases}
\end{equation*}
\end{thm}

In particular, we highlight that we improved the bound for subcolouring squares of planar graphs from $135$ to $43$. As a byproduct, we also improve the bound for cubes of planar graphs from 364 to 95.
\begin{thm}
\label{planarcube}
If $G$ is a planar graph, then $\chisub(G^{3}) \leq 95$. 
\end{thm}

 To obtain Theorem \ref{thm:girthbounds}, we rely on techniques recently developed by Almulhim and Kierstead  \cite{ALMULHIM2022112631, almulhim2020thesis} to bound (low length) generalised colouring numbers of planar graphs. These techniques make use of and improve on the decompositions of triangulated planar graphs given in~\cite{vandenHeuveletal2017} to obtain Theorem~\ref{thm:wcolplanar}. These decompositions are  the precursors of the graph product structure theorem of planar graphs~\cite{productstructuretheorem}. In order to obtain Theorem \ref{thm:girthbounds},  we extend these decompositions beyond triangulated planar graphs, to planar graphs of larger girth (see Theorem \ref{largegirthisometricpath}).

\begin{table}[h]
\begin{center}
\begin{tabular}{|c|c|}
\hlx{hvv}
\textbf{Constraint on $G$}&\textbf{Upper bound for $\wcol_{\ell}(G)$}\\
\hlx{vvhhvv}
treewidth at most $t$&$\wcol_{\ell}(G)\leq\binom{t+\ell}{t}$ \cite{grohe2018coloring}\\
\hlx{vvhvv}
simple treewidth at most $t$&$\wcol_{\ell}(G)\leq (\ell+1)^{t-1}(\lceil \log \ell \rceil +2)$ \cite{JoretMicek2022}\\
\hlx{vvhvv}
genus at most $g$&$\wcol_{\ell}(G)\leq \bigl(2g+\binom{\ell+2}{2}\bigr)(2\ell+1)$ \cite{vandenHeuveletal2017}\\
\hlx{vvhvv}
$K_t$-minor free, $t\geq4$&$\wcol_{\ell}(G)\leq \binom{t+\ell-2}{t-2}(t-3)(2\ell+1)$ \cite{vandenHeuveletal2017}\\
\hlx{vvhvv}
$K^*_{2,t}$-minor free, $t\geq 2$& $\wcol_{\ell}(G)\leq (t-1)(\ell +1)(2\ell +1)$ \cite{vandenHeuvelWood2018}\\
\hlx{vvhvv}
$K^*_{3,t}$-minor free&  $\wcol_{\ell}(G)\leq (2t+1)\binom{\ell +2}{2}(2\ell +1)$ \cite{vandenHeuvelWood2018}\\
\hlx{vvhvv}
$K^*_{s,t}$-minor free, $t\geq 2$& $\wcol_{\ell}(G)\leq s(t-1)\binom{s+\ell}{s}(2\ell+1)$ \cite{vandenHeuvelWood2018}\\
\hlx{vvh}
\end{tabular}
\end{center}
\caption{Upper bounds on $\wcol_{k}(G)$ according to several constraints on $G$.}
\label{tab:weak}
\end{table}

We end the paper with algorithmic results.  We consider two settings. 

In a first setting, we assume that the input
is  a graph $G$ in some class and the power $d$ that we are going to take. Then, we compute an ordering $\sigma$ of $G$ which attains the bound we give for the generalised colouring numbers, and use the fact that Theorem \ref{thm:gcol} is algorithmic to give a subcolouring of the $d^{th}$-power of the graph $G$ in polynomial time. 

In a second setting, we assume that the input is a graph $H = G^{d}$, where $G$ belongs to a (known) bounded expansion class. However,
 neither the underlying graphs $G$ nor the integer $d$ is given. In this case, it is not obvious how to find a proper subcolouring, even though we know one with a bounded number of colours exists. We manage to give a $2$-approximation to this problem in some cases. We need some definitions to state the the class of graphs for which the algorithm applies. Let $G$ be a graph, a \emph{layering} of $G$ is a sequence of disjoint sets $(L_{0},L_{1},\ldots,L_{t})$ such that for all $v \in V(G)$, we have $v \in L_{i}$ for some $i \in \{0,\ldots,t\}$, and if $uv \in E(G)$ where $u \in L_{i}$ and $v \in L_{j}$, we have $|i-j| \leq 1$. We say a class of graphs $\mathcal{C}$ has \emph{bounded layered treewidth} if there exists an integer $k$ such that for all graphs $G \in \mathcal{C}$, there exists a tree decomposition $(T,\beta)$ of $G$ and a layering $L = (L_{0},\ldots,L_{t})$ of $G$ such that each bag of $(T,\beta)$ intersects a layer in at most $k$ vertices. Similarly, we say that a graph class $\mathcal{C}$ has bounded layered cliquewidth if for every graph $G \in \mathcal{C}$, there exists a layering $L = (L_{0},\ldots,L_{t})$ of $G$ such that each $L_{i}$ induces a graph with bounded cliquewidth. We say $\mathcal{C}$ has \emph{algorithmic bounded layered cliquewidth} if $\mathcal{C}$ has bounded layered cliquewidth, and further a layering can be found in polynomial time. 

\begin{thm}
\label{thm:approximationalgo}
If $\mathcal{C}$ is a graph class with algorithmic bounded layered cliquewidth, then there exists a $2$-approximation algorithm for the subchromatic number of graphs in $\mathcal{C}$.
\end{thm}

Planar graphs have bounded layered treewidth (see \cite{productstructuretheorem}, Theorem 11) and in particular, every breadth-first search layering of a planar graph gives rise to a layering  which witnesses this. In Lemma~\ref{lem:clarify} we argue that this implies that  powers of planar graphs have bounded layered cliquewidth, and such a layering can be computed in polynomial time. Thus a special case of Theorem \ref{thm:approximationalgo} says we can $2$-approximate the subchromatic number of powers of planar graphs, without knowledge of the underlying planar graph or the power. We prove this theorem using an algorithm that is similar to the well-known Baker's algorithm. We take a breadth-first search tree in the graph, and  its layers have bounded cliquewidth. Then we compute the subchromatic number exactly on each layer, and by using the same colours on odd and respectively even layers, gives a $2$-approximation. In fact, with only minor modifications, the algorithm can be used to compute $(p+1)$-shrubdepth covers of strongly local transductions of bounded expansion classes that have locally bounded treewidth and such that the bounded cliquewidth layering of the transduction is computable in polynomial time, partially answering a question posed in \cite{SBETOCL}, which asks whether such a shrubdepth cover can be computed in polynomial time for every first order transduction and any bounded expansion class (see \cite{SBETOCL} for definitions).

The paper is structured as follows.
In Section \ref{clusteringsection} we prove Theorem~\ref{thm:gcol} and obtain Theorem \ref{planarcube}. In Section \ref{hardstuff} we prove the decomposition theorem for arbitrary planar graphs and obtain Theorem \ref{thm:girthbounds} and our decompositions for planar graphs of arbitrary girth. We then obtain bounds for $\gcol_{k,\ell}(G)$ for various minor closed classes in Section~\ref{sec:manybounds}. Finally, in Section~\ref{sec:approx} we prove Theorem \ref{thm:approximationalgo}, and give some further remarks in Section \ref{sec:fur}.

\section{Clusterings and the semi-weak colouring number}
\label{clusteringsection}

In this section we prove Theorem \ref{thm:gcol}. For future purposes we prove a stronger version, namely Theorem \ref{semiweak}, which bounds the subchromatic number of the $d$-th power of a graph $G$ by its $(2d-1)$-th semi-weak colouring number $\semiweak_{2d-1}(G)$ (see definition below).  From this stronger result, Theorem~\ref{planarcube} directly follows, as
$\semiweak_5(G)\leq 95$ is known to hold for every planar graph $G$ \cite[Theorem 3.1.1]{almulhim2020thesis}.

Before we prove Theorem \ref{semiweak}, we need a lemma, which requires further definitions. A \emph{clustering} $\mathcal{X}$ of $G$ is a partition of $V(G)$ into disjoint sets (called \emph{blocks}) such that each block induces a clique. Given a graph $G$ and a clustering $\mathcal{X}$, the \emph{quotient graph $G/\mathcal{X}$} has vertex set the blocks of $C$, and two blocks $A,B$ are adjacent in $G / \mathcal{X}$ if there exists a vertex $v \in A$ and a vertex $u \in B$ such that $uv \in E(G)$.   

We also need another variant of the generalised colouring numbers, and further related notation.

For some particular cases of the indices, we shall use standard notation: So,  we will use the notation $\Reach_{\ell}(G,\sigma,y)$ to refer to $\GReach_{1,\ell}(G,\sigma,y)$ 
(or simply $\Reach(G,\sigma,y)$, if $\ell=1$);
we will use the notation $\WReach_{\ell}(G,\sigma,y)$ to refer to $\GReach_{\ell,\ell}(G,\sigma,y)$. Also, $\col_1(G)$, which is  the \emph{colouring number} of $G$, will be denoted by $\col(G)$, as usual.

We also define the set $\SubReach_{k}(G,\sigma,y)$ so that $x\in\SubReach_{k}(G,\sigma,y)$ if there exists an $y,x$-path $P_x=z_0,...,z_s$ with $z_0=x$, $z_s=y$ and $s\leq k$ such that it satisfies that $x$ is the minimum in $P_x$ and for every $\ceil{\frac{1}{2}k}\leq i\leq s$ we have $y\leq_{\sigma}z_i$. In this case we say that $y$ \emph{semi-weakly $k$-reaches} $x$. Then, we define $\subcol_{k}(G,\sigma)$ as $\subcol_{k}(G,\sigma):=\max_{y\in V}|\SubReach_{k}(G,\sigma,y)|$. Finally, we define the \emph{semi-weak $k$-colouring number} $\subcol_{k}(G)$ as the minimum $\subcol_{k}(G,\sigma)$ for $\sigma$ ranging over all the linear orderings of $V(G)$. Note that \begin{equation}\label{eq:semi}
    \semiweak_{\ell}(G) \leq \gcol_{\lceil \frac{\ell} {2} \rceil,\ell}(G)
\end{equation}
for any integer $\ell$. Note also that for odd values of $\ell$, $\semiweak_{\ell}(G)$ coincides with a parameter studied in \cite{vandenHeuveletal2019,almulhim2020thesis} to bound the chromatic numbers of exact distance graphs.

Now we can state our lemma.

\begin{lemma}\label{cluster}
Let $G$ be a graph, $\sigma$ an ordering of $V(G)$ and $d\geq2$ an integer. Then there exists a clustering $\mathcal{X}$ of $G^{d}$ that satisfies:
$$\col(G^{d}\!/\mathcal{X})\leq \subcol_{2d}(G,\sigma).$$
Moreover, if $d$ is odd then we have
$$\col(G^{d}\!/\mathcal{X})\leq \subcol_{2d-1}(G,\sigma).$$
\end{lemma}

\begin{proof}
For $u\in V(G)$ we define $m(u)$ as the minimum vertex with respect to $\sigma$  in $N^{\floor{\frac{d}{2}}}[u]$. 

Consider the equivalence relation $\sim$ on $V(G)$ given by $u\sim v$ if and only if $m(u)=m(v)$. We define the blocks of $\mathcal{X}$ as the 
equivalence classes of $\sim$.
\begin{claim}
    $\mathcal{X}$ is  a clustering of $G^{d}$. 
\end{claim}
\begin{claimproof}
We need to prove that  each block of $\mathcal{X}$ is a clique of $G^d$. Let
$u$ and $v$ belong to a same block (i.e. $u\sim v$). Then, 
$$d_{G}(u,v)\leq d_{G}(u,m(u))+d_{G}(m(u),v)=d_{G}(u,m(u))+d_{G}(m(v),v)\leq\floor{\frac{d}{2}}+\floor{\frac{d}{2}}\leq d.$$
Thus, $uv\in E(G^d)$.    
\end{claimproof}

For every $A\in V(G^{d}\!/\mathcal{X})$, we pick any vertex $u$ in $A$ and define $m(A)=m(u)$. Note that, by the definition of the equivalence relation $\sim$, the choice of vertex $u\in A$ does not matter.

We now define in $G^{d} \!/\mathcal{X}$ the ordering $\tau$ by $A<_{\tau}B$ if  $m(A)<_{\sigma}m(B)$. Note that $m(A)\neq m(B)$ whenever $A\neq B$, so $\tau$ is indeed a linear ordering. For a set $X\subseteq V(G^{d} \!/\mathcal{X})$ we let $m(X)=\{m(A) \colon A\in  X \}$. Indeed, for every such $X$, $|m(X)|=|X|$.

\begin{claim}
Let $h=\floor{d/2}$. 
Then,
$$m(\Reach[G^{d}\!/\mathcal{X},\tau,A])\subseteq \SubReach_{d+2h}[G,\sigma,m(A)].$$      
\end{claim}

\begin{claimproof}

Let $B\in\Reach[G^{d}\!/\mathcal{X},\tau,A]$. If $B=A$, then we have $m(B)=m(A)\in \SubReach_{d+2h}[G,\sigma,m(A)]$. Thus, we can assume $B<_{\tau}A$ and $BA\in E(G^{d}\!/\mathcal{X})$. Note that $m(B)<_{\sigma} m(A)$ because $B<_{\tau}A$. 
Since $BA\in E(G^{d}\!/\mathcal{X})$, there exists $b\in V(B)$, $a\in V(A)$ such that $ba\in E(G^{d})$. Thus, there exists an $b,a$-path $Q=v_0v_1...v_q$ in $G$ with $v_0=b$, $v_q=a$ and $1\leq q\leq d$.

By definition of $m(B)$ there is an $m(B),b$-path $P=u_0u_1...u_p$ in $G$ with $u_0=m(B)$, $u_p=b$ and $1\leq p\leq h$. In a similar way, there is an $a,m(A)$-path $R=w_0w_1...w_r$ with $w_0=a$, $w_r=m(A)$ and $1\leq r\leq h$.

We now prove that 
the $m(B),m(A)$-walk $PQR$ contains a path $T$ that witnesses that $m(B)\in \SubReach_{d+2h}[G,\sigma,m(A)]$.

Since the walk has length at most $d+2h$  any $m(B),m(A)$-path contained in it will also have at most this length. Moreover we can pick one such path $T=T_1T_2T_3$, where $T_1$ is a path in $P$, $T_2$ a path in $Q$ and $T_3$ a path in $R$. Let $t_0...t_s$ be the sequence of vertices of $T$, with $t_0=m(B), t_s=m(A)$. Note that for each $i\in [h]$ we have that $m(b)=m(B)\leq_{\sigma} v_i$ and as $v_{q-i}\in N^{h}[a]$ we get $m(A)\leq_{\sigma} v_{q-i}$. Moreover, for every $i \in [p]$ we have $m(B) \le_\sigma u_i$ and for every $i \in [r]$ as $w_{i-1}\in N^{h}[a]$ we have $m(a)=m(A)\leq_{\sigma}w_{i-1}$. Since $m(B)<_\sigma m(A)$, every vertex $t\in T$ satisfies $m(B)\leq_{\sigma}t$. 

As the last $r+h+1$ vertices of the walk $PQR$ belong to $N^h[a]$, only the first $p+q-h\le d$ vertices of this walk are possibly 
out of $N^h[a]$. It follows that only the first $p+q-h$ vertices of $T$ can possibly be out of $N^h[a]$. Hence, 
for $d\leq i\leq s$ we have
$t_i\in N^h[a]$, and thus $t_i\geq_\sigma m(A)$.
As $\ceil{\frac{d+2h}{2}}=d$, we deduce that $T$ witnesses that we have
$m(B)\in \SubReach_{d+2h}[G,\sigma,m(A)]$.
\end{claimproof}



From this claim, we get that for  every $A\in\mathcal X$ we have 
$$|\Reach[G^{d}\!/\mathcal{X},\tau,A]|=|m(\Reach[G^{d}\!/\mathcal{X},\tau,A])|
\leq |\SubReach_{d+2h}[G,\sigma,m(A)]|.$$

Thus, $\col(G^{d}\!/\mathcal{X},\tau)\leq \subcol_{d+2h}(G,\sigma)$.
\end{proof}

From this we can deduce our version of Theorem \ref{NOdMPZ}. Note that Theorem \ref{thm:gcol} follows from the following theorem and \eqref{eq:semi}.

\begin{thm}
\label{semiweak}
For any graph $G$, and any fixed integer $d \in \N$, we have $\chisub(G^d) \leq \semiweak_{2d}(G)$. Moreover if $d$ is odd, then we have $\chisub(G^d) \leq \semiweak_{2d-1}(G)$.
\end{thm}

\begin{proof}
Let $\ell=2d$ if $d$ is even and $\ell=2d-1$ if odd. Let $\sigma$ be an ordering of $V(G)$ that witnesses $\subcol_{\ell}(G,\sigma)=\subcol_{\ell}(G)$.  Lemma~\ref{cluster} guarantees there is a clustering such that $\col(G^{d}\!/\mathcal{X})\leq \subcol_{\ell}(G)$. Then we have $\chisub(G^d)\le \chi(G^d/\mathcal{X})\le \col(G^{d}\!/\mathcal{X})\leq \subcol_{\ell}(G)$, as desired.
\end{proof}

\begin{proof}[Proof of Theorem \ref{planarcube}]
Theorem 3.1.1 of \cite{almulhim2020thesis} states that $\semiweak_{5}(G)\le 95$ for every planar graph~$G$. Our result then follows from Theorem \ref{semiweak}.
\end{proof}

\section{Bounding the semi-weak-colouring number of planar graphs}
\label{hardstuff}
In this section we prove Theorem \ref{thm:girthbounds}. This result follows immediately from Theorem \ref{semiweak} and the following theorem, to which we dedicate this section. 

\begin{thm}
\label{boundingsemiweak}
For any planar graph $G$, $\semiweak_{4}(G) \leq 43$. Further, if we let $g$ be the girth of $G$, then 
\begin{equation*}
\semiweak_{4}(G)\leq
    \begin{cases}
        39 & \text{if } g \geq 10 \\
        15 & \text{if } g \geq 17 \\
    \end{cases}
\end{equation*}
\end{thm}

We follow the approach in \cite{ALMULHIM2022112631} to bound $\wcol_2(G)$ for $G$ planar. This approach builds on techniques developed in~\cite{vandenHeuveletal2017}, and is also used in \cite{almulhim2020thesis}.

As notation, we will use $d_{G}(x,y)$ to refer to the length of the shortest path in $G$ from $x$ to~$y$. A path $P$ in a graph $G$ is \emph{isometric} if there is no shorter path in $G$ between the endpoints of $P$. For a path $P$, let $vPx$ denote the subpath in $P$ from $v$ to $x$. We  say two vertex disjoint paths $P$ and $P'$ are \emph{adjacent} if there exists a vertex $v \in V(P)$ and $v' \in V(P')$ such that $vv' \in E(G)$. 
When two vertex disjoint subpaths $xPy$ and $uP'v$ of two paths $P,P'$ have their endpoints $y$ and $u$ adjacent, we denote by $xPyuP'v$ the path from $x$ to $v$ obtained by concatenating them.
Observe that isometric paths are induced subgraphs, and that they satisfy the following easy observation \cite{vandenHeuveletal2017}, of which we recall a short proof for convenience.

\begin{lemma}[Van den Heuvel et al. \cite{vandenHeuveletal2017}]
\label{naivebound}
Suppose $H$ is a graph, $P$ is an isometric path in $H$,  $x,y \in V(P)$, and $v\in V(H)$.
If $d_{H}(v,x) \leq r$ and $d_{H}(v,y) \leq r$, then $d_{P}(x,y) \leq 2r$. In particular, there are at most $2r+1$ vertices with distance at most $r$ from $v$ in $P$. 
\end{lemma}
\begin{proof}
Let $P$ be an $u-w$ isometric path in $H$ and $xPy$ be the subpath of $P$ that connects $x$
and $y$. We note that there is a $x-y$ walk $W_{x,y}$ in $H$ of length at most $2r$ that passes through the vertex $v$ 
and so there is a $x-y$ path $P'_{xy}$ in $H$ of length at most $2r$. Suppose $d_P(x,y)>2r$
in $P$, then we replace the subpath $xPy$ in $P$ by $P'_{xy}$ and get an $u-w$ walk of length smaller than the length of $P$ and so there is another $u-w$ path of length smaller than $P$ in $H$.
This contradicts the assumption that the path $P$ is an isometric path in $H$.  

Without loss of generality, let $x$ and $y$ be the minimum and the maximum-index vertices in $P$ respectively, such that 
the distances $d(x,v),d(v,y)$ are at most $2r$. If there are more than $2r+1$  
many vertices with distance at most $r$ from $v$ in $P$, then $d_P(x,y)$ is more than $2r$, a contradiction.    
\end{proof}

For a planar graph $G$ we will create a vertex ordering $\sigma$ in the following manner: we take an isometric path, put its vertices at the start of the ordering and remove it from the graph. Then, we pick a path which is isometric in the remaining graph, remove it and put it next in the ordering. We proceed inductively in this way, until no vertices are left to be ordered. This motivates the following definition from \cite{vandenHeuveletal2017}. We say a \emph{decomposition} of a graph $G$ is a set $\mathcal{H} = \{H_{1},\ldots,H_{s}\}$ for some integer $s \in \mathbb{N}$, where 
$H_{i}$ is an induced subgraph of $G$ (for $1\leq i\leq s$), $V(G) = \cup_{i=1}^{s} V(H_{i})$, and $V(H_{i}) \cap V(H_{j}) = \emptyset$ if $1\leq i< j\leq s$. We let $\mathcal{H}_{i} = G- \cup_{j=1}^{i-1} V(H_{j})$. 

An \emph{isometric-path decomposition} $\mathcal{P}=\{P_0, \dots, P_s\}$ is a decomposition where every subgraph $P_i$ is an isometric path in $\mathcal{P}_i$. 
 We will not be happy with just any isometric-path decomposition, but will need a particular type called ``reductions''. The reductions from \cite{vandenHeuveletal2017} help us bound the number of paths a vertex can reach, but to further bound $\semiweak_4(G)$ we make a more specific choice of a reduction, following \cite{ALMULHIM2022112631}.

\begin{definition}~\label{def:reduction}
A \emph{reduction} of a triangulated planar graph $G$ is an isometric-path decomposition $\mathcal{P} = \{P_{0},\ldots,P_{s}\}$ of $G$ such that:
\begin{enumerate}
\item{The isometric path $P_{0}$ consists of two vertices, and $P_{1}$ consists of a single vertex.}
\item{For all $i \in \{0,\ldots,s\}$, the path $P_{i}$ has endpoints $w_{i},w'_{i}$ (possibly these are equal), and for all $k \in  \{2,\ldots,s\}$, $P_{k}$ is adjacent to exactly two paths $P_{h}$ and $P_{j}$ with $h < j <k$, where there are some $v_{k},v_{k}' \in V(P_{h})$, $z_{k},z_{k}' \in V(P_{j})$, such that $v_{k}z_{k}w_{k}$ and $v_{k}'w_{k}'z_{k}'$ both bound faces in $G$. 
\item If there are two possible choices for $P_{k}$, then we choose the path which minimises the number of vertices in the interior of $v_{k}P_{h}v_{k}'w'_{k}P_{k}w_{k}$.}
\item{For every component $K$ of $\mathcal{P}_{k+1}$, the boundary of the face of $G[V(P_{0}) \cup \cdots \cup V(P_{k})]$ containing $K$ is of the form $D = vP_{h}v'z'P_{j}zv$ for some $h < j\leq k$. }
\end{enumerate}
\end{definition}

\begin{figure}[ht]
    \centering
    \includegraphics[width=.7\textwidth]{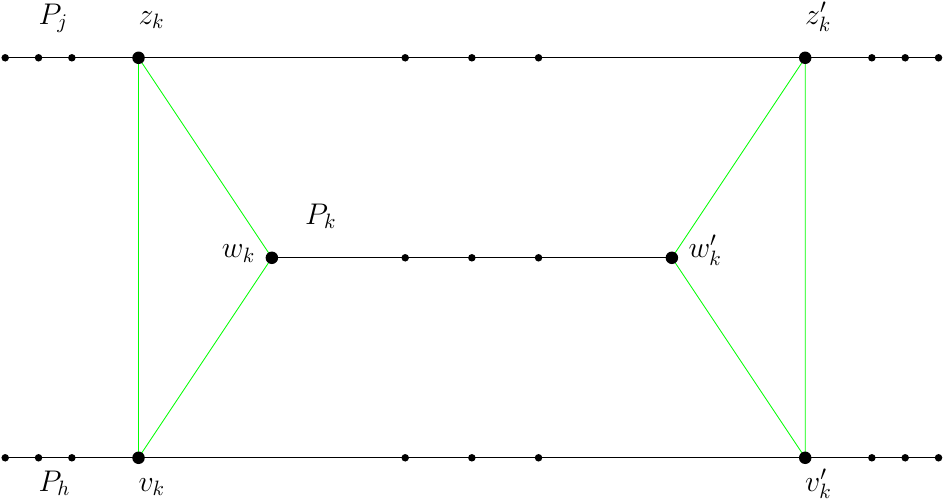}
    \caption{A path $P_k$ according to Definition~\ref{def:reduction}}
    \label{fig:isom}
\end{figure}

In \cite{vandenHeuveletal2017} it is shown that every triangulated planar graph has a reduction (although the choice in point (3) was not included in that paper, and was first considered in \cite{almulhim2020thesis,ALMULHIM2022112631}).

\begin{lemma}[Van den Heuvel et al. \cite{vandenHeuveletal2017}]\label{triangulated}
Every triangulated planar graph has a reduction.
\end{lemma}

We also need to introduce reductions for planar graphs of larger girth. We cannot prove as strong of a result, but we are able to keep the essential property.

\begin{definition}
A \emph{neat reduction} of a planar graph $G$ is an isometric-path decomposition $\mathcal{P} = \{P_{0},\ldots,P_{t}\}$ of $G$ satisfying that for $i\ge 1$ each path $P_{i}$ is adjacent to at least one and  at most two paths $P_{j}$ where $j <i$.
\end{definition}

\begin{thm}
\label{largegirthisometricpath}
If $G$ is a connected planar graph with girth  $g\ge 3$, then $G$ contains a neat reduction.
\end{thm}

\begin{proof}
Let $G$ be a planar graph with girth exactly $g$. We build our isometric path decomposition iteratively. We will maintain the following stronger property: If we have obtained paths $P_{0},\ldots,P_{\ell}$, $\ell \geq 1$, of our decomposition, and $K$ is a component of $\mathcal{P}_{\ell+1}$, then there is at least one and at most two of these paths that contain vertices that are adjacent, in $G$, to vertices in (the outerface of) $K$. 
Let $C$ be the boundary of the outerface of $G$.

Let $x,y$ be any pair of adjacent vertices in $V(C)$ and let $P_{0} = x,y$. Let $x'$ be a vertex adjacent to $x$ in $C$ and if it exists, let $y'$ be a vertex adjacent to $x'$ in $C$, and otherwise have $y' = x'$. Then let $P_{1} = x',y'$.

Now suppose we have found paths $P_{0},\ldots,P_{i}$, $i \geq 1$, such that for all $\ell \in \{1,\ldots,i\}$, we have that $P_{\ell}$ is isometric in $\mathcal{P}_{\ell}$, and is adjacent to at most two paths $P_{q}$ where $q < \ell$. Let $K$ be a component of $\mathcal{P}_{i+1}$ (if no such component exists, we are done). By our assumption we know that the facial boundary of $K$ is adjacent to at most two paths in $\{P_{0},\ldots,P_{i}\}$. We consider cases.

First suppose that $K$ is adjacent to two distinct paths, $P_{h}$ and $P_{j}$. We first consider the case where $P_{h}$ is incident to at least two vertices of $K$. Let $x,y$ be two vertices in $P_{h}$ that are incident to vertices in $K$ where $d_{P_{h}}(x,y)$ is maximized (note, it is possible that $x =y$). Let $x',y'$ be vertices in $K$ that are incident to $x,y$ respectively, and let $P$ be any isometric path between $x'$ and $y'$ in $K$, and further pick $x',y'$ such that the interior of $xP_{h}yy'Px'$ is maximized. We claim that letting $P_{i+1} = P$ maintains the desired property. Note that $P$ is isometric in $\mathcal{P}_{i+1}$ by construction and that it is adjacent at least to $P_h$. Further note that any component in $\mathcal{P}_{i+1}$ that was not adjacent to both $P_{h}$ or $P_{j}$ is still adjacent to at most two paths in $\{P_{0},\ldots,P_{i+1}\}$. If a component $K'$ in $\mathcal{P}_{i+1}$ was adjacent to both $P_{h}$ and $P_{j}$, then $K'$ is distinct from $K$ and $K'$ is still only adjacent to $P_{h}$ or $P_{j}$ (if $K'$ had an edge to $P$, this would imply $K' = K$). Lastly, if $K'$ is a new component created by the deletion of $P$, if $K'$ lies in the interior of $xP_{h}yy'Px'$, then it is adjacent to at most $P$ and $P_{h}$. If $K'$ lies on the exterior, it is adjacent to at most $P_{j}$ and $P$, as if such a component was adjacent to $P_{h}$, it would contradict that we picked $x'$ and $y'$ such that the interior of $xP_{h}yy'Px'$ is maximized or that $x,y$ were picked so as to maximise $d_{P_{h}}(x,y)$. 

Therefore by symmetry we may assume that $P_{h}$ is adjacent to exactly one vertex of $K$, and $P_{j}$ is adjacent to exactly one vertex of $K$. Let $x$ be the vertex in $K$ incident to a vertex in $P_{h}$, and $y$ be the vertex in $K$ incident to a vertex in $P_{j}$. Note $x$ may equal $y$. Let $P$ be any isometric path in $K$ from $x$ to $y$. Let $P_{i+1} = P$. In this case, again the condition holds as any new component created is adjacent to at least $P$, and any previous component is adjacent to the same paths as before. 

Therefore we may assume that $K$ is adjacent to at most one path $P_{h}$. Suppose first that $P_{h}$ is adjacent to two vertices $x,y$ in $K$. Let $P$ be any isometric path between $x$ and $y$. Let $P_{i+1} = P$. Observing that any new component is adjacent to at most $P$ and $P_{h}$ gives us the desired properties. Finally, if $P_{h}$ is adjacent to exactly one vertex $x$ in $K$, then pick any vertex $y \in V(K)$ (possibly $x=y$) and let $P$ be an isometric $(x,y)$-path in $K$. Then setting $P_{i+1} = P$, using the same reasoning as in the previous cases, the desired condition is satisfied, completing the proof. 
\end{proof}

With these new reductions we are ready to prove Theorem \ref{boundingsemiweak}. 

\begin{proof}[Proof of Theorem \ref{boundingsemiweak}]
Let $G$ be any planar graph with girth $r$ for some integer $r$. We may assume that $G$ is connected. If $r <10$, then as adding edges to $G$  cannot decrease $\semiweak_4(G)$, so we may assume $G$ is a triangulation, and thus by Lemma~\ref{triangulated}, $G$ has a reduction $\mathcal{P} = \{P_{0},\ldots,P_{s}\}$. Otherwise, we take $\mathcal{P} = \{P_{0},\ldots,P_{s}\}$ to be a neat reduction of $G$ guaranteed by Theorem \ref{largegirthisometricpath}. 

Following the notation of \cite{ALMULHIM2022112631}, if for $i\ge 2$, $P_{i}$ is adjacent to $P_{h}$ and $P_{j}$ for $h < j < i$, then we say $P_{h}$ and $P_{j}$ are the \emph{bosses} of $P_{i}$, and in particular, $P_{h}$ is the \emph{manager} of $P_{i}$, and $P_{j}$ is the \emph{foreman} of $P_{i}$. If $P_{i}$ is adjacent to only one path $P_{h}$, then we say that $P_{h}$ is a boss and manager of $P_{i}$.

We construct an ordering $\sigma$ on $V(G)$ by first saying that if $v \in V(P_{i})$ and $v' \in V(P_{j})$, where $i < j$, then $v <_{\sigma} v'$. Further if $P_{i} = w^{0}w^{1}\dots w^{n}$ we let $w^{j} <_{\sigma} w^{k}$ if $j < k$. 

To prove Theorem \ref{boundingsemiweak}, it suffices to show that $\semiweak_4(G,\sigma)$ is at most $43$ (or smaller if the girth of $G$ is larger than 9). Fix some path $P_{k} \in \mathcal{P}$, and let $v \in V(P_{k})$. We will prove that $|\SubReach_{4}(G,\sigma,v)|$ satisfies the desired bound. If $P_{k} = P_{0}$, then in the case of a reduction $|V(P_{0})| =2$, there is at most one vertex $x$ with $x <_{\sigma} v$, and thus $|\SubReach_{4}(G,\sigma,v)| \leq 2$. Otherwise, $|\SubReach_{4}(G, \sigma,v)| \leq 5$, as all vertices that can be reached are on a path.  If $P_{k} = P_{1}$, then in the case of a reduction as $|V(P_{1})| = 1$, we have that there are at most two vertices $x$ such that $x <_{\sigma} v$, so $|\SubReach_{4}(G,\sigma,v)| \leq 3$. Otherwise, 
the girth is at least $4$, and $|\SubReach_{4}(G,\sigma,v)| \leq 9$ as a vertex can only reach at most four vertices on each of the paths, as well as itself.  Therefore we assume that $k \geq 3$.  Let $P_{j}$ and $P_{h}$ be the bosses of $P_{k}$ ( $P_{j} = P_{h}$ is possible). If $P_k$ does have two bosses, without loss of generality, we say that $P_{h}$ is the manager of $P_{k}$, and $P_{j}$ is the foreman. Thus, in this case and if $G$ is triangulated, by property (3) of Definition~\ref{def:reduction} $P_{h}$ is a boss of $P_{j}$. Let $P_{i}$ be the other boss of $P_{j}$ in the triangulated case,  in the larger girth case, let $P_{q}$ be a boss of $P_{j}$, and let $P_{g}$ and $P_{f}$ be the two bosses of $P_{h}$. Note these paths need not all be distinct. In particular, it is possible that $P_{f} = P_{i}$, or $P_{g} = P_{i}$. 
For $a \in \{k,h,j,f,g,i,q\}$, we let $W_{a} = \{V(P_{a}) \cap \SubReach_{4}(G,\sigma,v)\}$. We encourage the reader to use Figure \ref{pathpic} as a guide for the following claims.
\begin{figure}
\centering
\includegraphics[scale = 0.5]{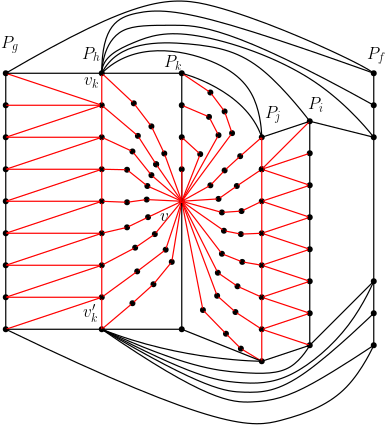}
\caption{A possible set of paths $P_{g},P_{h},P_{k},P_{j},P_{i},P_{f}$ for a triangulated graph (some edges omitted) and a vertex $v$. In this picture, we have $|W_{g}| = |W_{h}| = 9$, $|W_{k}| = 5$, $|W_{j}| = 8$, $|W_{i}| = 8$, and $|W_{f}| = 0$. Edges in red are part of a path that shows some vertex $u$ is $4$-semi-weakly reachable from $v$ (in some cases, more than one such path exists).}
\label{pathpic}
\end{figure}
\begin{claim}\label{claim:allpaths}
We have that $\SubReach_{4}(G,\sigma,v) \subseteq W_{i} \cup W_{k} \cup W_{h} \cup W_{j} \cup W_{f} \cup W_{g} \cup W_{q}$.
\end{claim}

\begin{claimproof}
Suppose that $u \in \SubReach_{4}(G,\sigma,v)$ and let $Q$ be a $(v,u)$-path which witness this. First suppose that in $Q$, the only vertex $x$ such that $x \leq_{\sigma} v$, is $u$. As $P_{h}$ and $P_{j}$ are the bosses of $P_{k}$, it follows that $u \in W_{k} \cup W_{h} \cup P_{j}$. Thus we only need to consider the case where there are two vertices $x_{1},x_{2} \in V(Q)$ that are smaller than $v$ under $\sigma$. In this case, by the same reasoning as above, we may assume that $x_{1} \in V(P_{h}) \cup V(P_{j}) \cup V(P_{k})$ and that $u = x_{2}$. Thus if $u$ does not lie in one of $P_{h}$, $P_{j}$,  or $P_{k}$, it lies in a boss of $P_{h}$, $P_{k}$ or $P_{j}$. These are precisely $P_{g},P_{f},P_{q}$ and $P_{i}$ by definition, and thus the claim follows. 
\end{claimproof}

For the rest of the proof, we bound the sizes of the sets $W_{k},W_{h},W_{j},W_{f},W_{g},W_{q}$ and $W_{i}$. 

\begin{claim}
\label{naiveboundclaim}
For any $a \in \{k,h,j,f,g,i,q\}$, we have $|W_{a}| \leq 9$
\end{claim}

\begin{claimproof}
Note $P_{a}$ is isometric in $\mathcal{P}_{a}$, and as $v \in V(\mathcal{P}_{a})$ by applying Lemma \ref{naivebound} with $r =4$, we get that $|W_{a}| \leq 9$. 
\end{claimproof}

\begin{claim}
\label{Pkbound}
We have that $|W_{k}| \leq 5$. Further if the girth of $G$ is at least $8$, then $|W_{k}| \leq 4$, and if the girth of $G$ is at least $9$, then $|W_{k}| \leq 3$. 
\end{claim}

\begin{claimproof}
Let $P_{k}=w^0w^1\dots w^n$ where $v=w^\ell$. Our definition of $\sigma$ implies that $v$ can reach itself, as well as possibly $w^{\ell-1},w^{\ell-2},w^{\ell-3}$ and $w^{\ell-4}$ (if they exist). These are the only such vertices $v$ can reach in $P_{k}$ by paths of length at most $4$ in $\mathcal{P}_{k}$, as $P_{k}$ is isometric in $\mathcal{P}_{k}$, and $v < w^{q}$ for all $q > \ell$ in $P_{k}$. If $G$ has girth at least $8$, then any path witnessing that $v$ semi-weakly $4$-reaches $w^{k-3}$ creates a cycle of length at most $7$, implying that $|W_{k}| \leq 4$. Similarly, if $G$ has girth at least $9$, then any path witnessing that $v$ semi-weakly $4$-reaches $w^{\ell-4}$ or $w^{\ell-3}$ creates a cycle of length at most $8$, implying that $|W_{k}| \leq 3$. 
\end{claimproof}

\begin{claim}
\label{distinctness}
If $G$ is triangulated and there exists a pair $a,b \in \{k,h,j,f,g,i\}$, such that $a \neq b$, but $P_{a} = P_{b}$, then $|\SubReach_{4}(G,\sigma,v_{k})| \leq 41$.
\end{claim}

\begin{claimproof}
If $G$ is triangulated we have $h=q$, and if we also have $P_{a} = P_{b}$ and $a \neq b$ for some pair $a,b \in \{k,h,j,f,g,i\}$, then the set $\{P_{k},P_{h},P_{j},P_{f},P_{g},P_{i}\}$ has at most five elements. Thus it follow from Claim \ref{claim:allpaths}, Lemma \ref{naivebound} and Claim \ref{Pkbound} that $|\SubReach_{4}(G,\sigma,v_{k})| \leq  4(9) + 5 = 41$.
\end{claimproof}

If $G$ is triangulated our goal is to show that $|\SubReach_{4}(G,\sigma,v)| \leq 43$, and by Claim \ref{distinctness} we may assume that for all $a,b \in \{k,h,j,f,g,i\}$, we have that $P_{a} = P_{b}$ if and only if $a = b$.

\begin{claim}\label{claim:hboss}
If $G$ is triangulated we have that $P_h$ is a boss of $P_i$, and thus $P_{h}$ is the manager of $P_{j}$. 
\end{claim}

\begin{claimproof}
Recall that $P_h$ and $P_i$ are the bosses of $P_j$ and that by property (3) of Definition~\ref{def:reduction} these two paths are adjacent. So if $P_h$ is not a boss of $P_i$, then $P_{i}$ is a boss of $P_{h}$. This implies that either $P_{f}$ or $P_{g}$ is $P_{i}$, contradicting our assumption. 
\end{claimproof}

\begin{claim}
We have that $|W_{j}| \leq 8$. Further for $a \in \{h,j,f,g,i, q\}$, if the girth of $G$ is at least $10$, then $|W_{a}| \leq 6$. If the girth of $G$ is at least $17$, then $|W_{a}| \leq 2$.
\end{claim}

\begin{claimproof}
By Lemma \ref{naivebound}, it follows that $|W_{j}| \leq 9$. For a contradiction we assume that $G$ is a triangulation and $|W_{j}| = 9$. Observe that for a single path $Q$ starting at $v$ there can be at most $3$ vertices of $W_{j}$ in the path. Thus if $|W_{a}| \geq 4$, then there are two distinct paths $Q_{1}$ and $Q_{2}$ starting at $v$ which certify vertices are in $W_{j}$. Suppose $Q_{1}$ and $Q_{2}$ have endpoints $x$ and $y$ respectively, and that we pick $Q_{1}$ and $Q_{2}$ such that the distance between $x$ and $y$ is maximized on $P_{j}$. As we assume that $|W_{j}| = 9$, we have that $d_{P_{j}}(x_{1},x_{2}) =8$, while by definition we have that $|E(Q_{1})| \leq 4$ and $|E(Q_{2})| \leq 4$. Further $Q_{1}$ and $Q_{2}$ are disjoint from $P_{j}$ except exactly at $x_{1}$ and $x_{2}$ and perhaps at the neighbour of $x_1$ in $P_j$ that is between $x_1$ and $x_2$. Then (using notation from Definition~\ref{def:reduction}), the path $w_{j}P_{j}x_1Q_{1}vQ_{2}x_2P_{j}w_{j}'$ is an isometric path in $\mathcal{P}_{j}$ with the same length and endpoints as $P_{j}$, and further reduces the number of vertices in the interior of the cycle formed with $P_{h}$. Since by Claim~\ref{claim:hboss} $P_h$ is the manager of $P_j$ this contradicts (3) of Definition~\ref{def:reduction}. Therefore, we have $|W_{j}| \leq 8$.   

Now suppose that the girth of $G$ is at least $11$, and for contradiction we have that $|W_{j}| \geq 7$. By Lemma \ref{naivebound} the maximum distance between two vertices in $W_{j}$ is $8$. Observe that if a vertex $x \in W_{j}$ is certified by a path $Q$ of length $4$ where $x$ is the only vertex in $P_{j} \cap V(Q)$, then the neighbours of $x$ in $P_j$ are not in $W_{j}$, as this would create a cycle of length at most $9$. Therefore if all vertices in $W_{j}$ can be certified by paths of length at most $4$ where the endpoint is the only vertex in $P_{j}$, we have that $|W_{j}| \leq 5$, as desired. Thus there must be at least one vertex $x \in W_{j}$ certified by a path $Q$ which has at most $3$ vertices not in $P_{j}$. We may assume that $x$ is an endpoint of $Q$, and $Q$ contains another vertex $y \in W_{j}$ (possibly by changing our choice of $x$ if needed). Note that $y$ must be a neighbour of $x$. If it exists, let $x_{1}$ be the neighbour of $x$ in $P_{j}$ that is not $y$, and $x_{2}$ be the neighbour of $x_{1}$ in $P_{j}$ that is not $x$ (again, if it exists). As $Q$ has at most $3$ vertices not in $P_{j}$, it follows that $x_{1}$ and $x_{2}$ are not in $W_{j}$, as otherwise we create a cycle of length at most $10$. As the maximum distance between two vertices in $V(P_{j}) \cap W_{j}$ is $8$, it now follows that $|W_{j}| \leq 6$. 


If the girth of $G$ is at least $17$, then if $|W_{j}| \geq 4$, there are at least two paths $Q_{1}$ and $Q_{2}$ which certify vertices in $W_{j}$. But as the maximum distance between two vertices in $W_{j}$ on $P_{j}$ is $8$, and $e(Q_{i}) \leq 4$ for $i \in \{1,2\}$, we have that these paths plus $P_{j}$ create a cycle of length at most $16$, a contradiction. Thus there is at most one such path, and such a path can certify at most $2$ vertices, implying that $|W_{j}| \leq 2$.  


Similar arguments work for the other $a \in \{h,j,f,g,i,q\}$.
\end{claimproof}

For the large girth cases, we already have our desired bound. If the girth of $G$ is at least $10$, then $\SubReach(G,\sigma,v) \leq 3 + 6\times 6 = 39$, as desired. 
Finally, if the girth of $G$ is at least $17$, then $\SubReach(G,\sigma,v) \leq 2 \times 6 +3 = 15$. Therefore for the rest of the section, we assume $G$ is triangulated.

\begin{claim}
We have that $|W_{i}| \leq 8$.
\end{claim}

\begin{claimproof}
As $P_{h}$ is the manager of $P_{j}$, if a path witnesses that $v$ semi-weakly $4$-reaches a vertex in $P_{i}$, then this path must intersect $P_{j}$. Let $Q_{1}$ and $Q_{2}$ be two paths starting at $v$ and ending at $x_{1},x_{2} \in E(P_{i})$ that witness that $v$ semi-weakly $4$-reaches $x_{1}$ and $x_{2}$ respectively, and such that both $Q_{1}$ and $Q_{2}$ intersect $P_{j}$ at vertices $y_{1},y_{2}$ respectively. By definition of $Q_{1}$ and $Q_{2}$, and as $y_{1}, y_2 \leq_\sigma v$, we have that $y_{1}x_{1} \in E(G)$, and $y_{2}x_{2} \in E(G)$. Thus $d_{P_{i}}(x_{1},x_{2}) \leq d_{P_{j}}(y_{1},y_{2}) + 2$, as otherwise we contradict that $P_{i}$ is isometric in $\mathcal{P}_{i}$ because the path from $y_{1}$ to $y_{2}$ in $P_{j}$ together with the edges $y_{1}x_{1}$ and $y_{2}x_{2}$, would form a shorter path. As the paths from $v$ to $y_{1}$ and $y_{2}$ have length at most $3$, this implies that $d_{P_{j}}(y_{1},y_{2}) \leq 5$, as otherwise either we contradict that $P_{j}$ is an isometric path, or it contradicts (2) in the definition of isometric path decomposition. Thus $d_{P_{i}}(x_{1},x_{2}) \leq 7$, and thus there are at most $8$ vertices that $v$ can reach in $P_{i}$ through vertices of $P_{j}$. 
\end{claimproof}

\begin{claim}
\label{Pi1}
Either $|W_{g}| \leq 6$ or $|W_{f}| \leq 6$. Further, if both $|W_{g}| \geq 4$ and $|W_{f}| \geq 4$, then $|W_{h}| \leq 7$. If both $|W_{g}| \geq 3$ and $|W_{f}| \geq 3$, then $|W_{h}| \leq 8$. If $|W_{h}| = 9$, then either $|W_{f}| =0$ or $|W_{g}| =0$.  
\end{claim}

\begin{claimproof}
As all of $P_{g},P_{h},P_{k},P_{j},P_{i}$ and $P_{j}$ are distinct, and by property (2) of isometric path decompositions, for one of $P_{g}$ or $P_{f}$ the vertices in this path can only be semi-weakly 4-reached by $v$ via the vertices $v_k$ and $v_k'$ of $P_h$ (here we are again using notation from Definition~\ref{def:reduction}). Without loss of generality, suppose that is so for $P_{f}$. Then each of $v_k$ and $v_k'$ are adjacent to a most $3$ vertices in $P_{f}$ by Lemma \ref{naivebound}. This implies that $|W_{f}| \leq 6$, as desired. Further, if $|W_{f}|\geq 4$, we have a path from $v$ to $v_{k}$ with length at most $3$, and one such path from $v$ to $v_{k}'$. This implies that the distance from $v_{k}$ to $v_{k}'$ in $P_{h}$ is at most $6$, as otherwise we contradict that $P_{h}$ is isometric in $\mathcal{P}_{h}$, and thus by Lemma \ref{naivebound} it follows that $|W_{h}| \leq 7$. Similarly, if $|W_{f}| =3$, then without loss of generality, there is a path of length at most three from $v$ to $v_{k}$. Let $w\in W_h$ be such that $d_{P_h}(w,v_k)$ is maximised. The distance from $v_{k}$ to $w$ is at most $7$, and thus by Lemma \ref{naivebound} we have $|W_{h}| \leq 8$. Finally, if $|W_{h}| =9$, then there is no path of length at most $3$ from $v$ to $v_{k}$ or $v_{k}'$, as otherwise we contradict that $P_{h}$ is isometric in $\mathcal{P}_{h}$. Therefore we have $|W_{f}| =0$ or $|W_{g}| =0$ in this case. 
\end{claimproof}

To finish the proof, we have to consider cases according to Claim~\ref{Pi1}. If $|W_{h}| = 9$, then without loss of generality we have $|W_{f}| =0$, and thus $|\SubReach_{4}(G,\sigma,v_{k})| \leq |W_{k}| + |W_{j}| + |W_{i}| + |W_{h}| + |W_{g}| + |W_{f}|  \leq 5 + 8\times 2 + 9\times 2   = 39$. So we may assume that $|W_{h}| \leq 8$. If $|W_{h}| =8$, then without loss of generality we may assume that $|W_{f}| = 3$, and in this case $|\SubReach_{4}(G,\sigma,v_{k})| \leq |W_{k}| + |W_{j}| + |W_{i}|+ |W_{h}| + |W_{g}| + |W_{f}|  \leq 5  + 8\times 3 + 9 + 3  = 41$. Therefore we may assume that $|W_{h}| \leq 7$. In this case we can assume without loss of generality we have $|W_{f}| \leq 6$. Therefore $|\SubReach_{4}(G,\sigma,v_{k})| \leq |W_{k}| + |W_{j}| + |W_{i}|+ |W_{h}| + |W_{g}| + |W_{f}| \leq 5+8\times2 + 7 +9 +6 =43$.
\end{proof}

\section{Improved bounds for powers of graphs with excluded minors}

Together with the bounds on the weak colouring numbers obtained in \cite{grohe2018coloring, vandenHeuveletal2017, vandenHeuvelWood2018, JoretMicek2022}, Theorem~\ref{NOdMPZ}  gives explicit upper bounds on $\chisub(G^p)$ when $G$ has bounded treewidth, bounded simple treewidth, bounded genus or excludes some minor. In this section we obtain upper bounds for $\gcol_{k,\ell}(G)$ when $G$ has bounded treewidth, bounded simple treewidth, bounded genus, or excludes some complete minor or some $K^*_{s,t}$ as a minor, where $K^*_{s,t}$ is the complete join of $K_s$ and $\bar K_t$. By Theorem~\ref{thm:gcol}, we obtain improved upper bounds for $\chisub(G^p)$ for all such $G$.


\begin{thm}\label{thm:gcol_bounds}
Let $k,\ell,t, g$ be positive integers with $k\le \ell$. For every graph $G$ we have the 
 upper bounds on  $\gcol_{k,\ell}(G)$ 
 displayed in Table~\ref{tab:gcol}, depending on the constraints on $G$.
\end{thm}

We dedicate the rest of the section to prove this result.

\subsection{Bounded treewidth}

A graph is a \emph{$k$-tree} if it is either a clique of order $k+1$ or can be obtained from a smaller $k$-tree by adding a vertex and making it adjacent to $k$ pairwise-adjacent vertices. The \emph{treewidth}, $\mathrm{tw}(G)$, of a graph $G$ is the smallest $k$ such that $G$ is a subgraph of a $k$-tree.

For a $k$-tree $G$, we say that an ordering $L$  of $V(G)$ is a \emph{simplicial ordering} if it is obtained in the following way. Fix a way of constructing $G$ from a $(k+1)$-clique $K_0$ and let the vertices of $K_0$ be the smallest in the ordering. Then for $v\notin K_0$ let $u<_{L}v$ if $u$ was added to the $k$-tree before $v$.

\begin{lemma}\label{lemma:reachInfinity}
Let $k$ be a positive integer, $G$ a $k$-tree, $L$ a simplicial ordering of $V(G)$. For every $v\in V(G)$ we have $\Reach_{1}[G,L,v]=\Reach_{\infty}[G,L,v]$
\end{lemma}
\begin{proof}
If $v$ is one of the smallest $k+1$ vertices in $L$ then every vertex $u<_Lv$ satisfies $u\in\Reach_1[G,L,v]$, as desired. Otherwise, consider the component $C_v$ of $(G\setminus\Reach_{1}[G,L,v])\cup\{v\}$ which contains $v$ and note that $v\leq_L w$ for every $w\in C_v$. If $x$ belongs to $\Reach_{\infty}[G,L,v]$ then every path that witnesses this has all its internal vertices in $C_v$. But if such a path exists then we must have $x\in\Reach_{1}[G,L,v]$ and the results follows. 
\end{proof}

The following can be deduced for the Theorem 4.2 of \cite{grohe2018coloring}.

\begin{lemma}\label{lemma:grohe}
Let $G=(V,E)$ be a graph and $L$ a linear ordering of $V$ with $t+1\ge \max_{y\in V}|\Reach_{\infty}[G,L,y]|$. For every positive integer $k$ and every $y\in V$ we have $$|\WReach_{k}[G,L,y]|\leq\binom{t+k}{t}.$$
\end{lemma}

In order to use this result for all generalised colouring numbers, and not just the weak colouring numbers, the following will be key.

\begin{lemma}
\label{GreachYWreach}
Let $k,\ell$ be positive integers with $k\leq\ell$, $G$ a $k$-tree, $L$ a  simplicial ordering of $V(G)$. For every $v\in V(G)$ we have
$$\WReach_{k}[G,L,v]=\GReach_{k,\ell}[G,L,v].$$
\end{lemma}
\begin{proof}

By definition we have $\WReach_{k}[G,L,v]\subseteq \GReach_{k,\ell}[G,L,v]$ so let us see that the other inclusion holds. Consider $u\in \GReach_{k,\ell}[G,L,v]\setminus\{v\}$. By definition, there exists an $u,v$-path $P=w_0...w_s$ with $w_0=v$, $w_s=u$, $s\leq\ell$, for every $i\in[s]$ $u<_{L}w_{i-1}$ and the set $I=\{j\in[s]:w_{j}\leq_{L}w_{i-1}\text{ for every } i\in[j]\}$ satisfies that $|I|\leq k$.

Note that $I$ is nonempty because we have $w_s=u\in I$. Sort the elements of $I$ in increasing order, that is let  $I=\{j_1,...,j_m\}$ where $j_i\leq j_{i+1}$ for every $i\in[m-1]$. By definition of $I$ we have that $w_{j_{i+1}}<_{L}w_{j_{i}}$ for every $i\in[m-1]$ and $w_{j_{m}}=u$.

Let us see that $T=vw_{j_{1}}...w_{j_{m-1}}u$ is an $u,v$-path that witnesses that $u\in \WReach_{k}[G,L,v]$. The subpath $v\dots w_{j_1}$ of $P$ witnesses that $w_{j_1}\in\Reach_{\infty}[G,L,v]$ which by Lemma $\ref{lemma:reachInfinity}$ implies $w_{j_1}\in\Reach_{1}[G,L,v]$. Similarly, for every $i\in[m-1]$ we have that $w_{j_{i}}w_{j_{i+1}}\in E(G)$, because some subpath of $P$ witnesses that $w_{i+1}\in\Reach_{\infty}[G,L,w_i]=\Reach_{1}[G,L,w_i]$. Since $u$ is minimum in $T$ with respect to $L$, the result follows. 
\end{proof}

We now obtain our upper bounds for graphs with bounded treewidth. Since $\gcol_{k,\ell}(G)$ cannot decrease if we add edges, we may assume that $G$ is a $k$-tree, and let $L$ be a simplicial ordering of $G$. Given that $G$ has treewidth at most $t$, the ordering $L$ satisfies that $t+1\geq |\Reach_{1}[G,L,v]|$ which by Lemma \ref{lemma:reachInfinity} implies $t+1\geq|\Reach_{\infty}[G,L,v]|$. Moreover, by Lemma \ref{GreachYWreach} for every vertex $v$ and for every $k$, we have $\WReach_{k}[G,L,v]=\GReach_{k,\ell}[G,L,v]$, and the result follows from Lemma \ref{lemma:grohe}.

\subsection{Bounded simple treewidth}

Suppose we build a $k$-tree with the restriction that when adding a new vertex, the clique to which we make it adjacent cannot have been used when adding some other vertex. In such a case, we say that the $k$-tree  is a \emph{simple $k$-tree}. The \emph{simple treewidth}, $\mathrm{stw}(G)$, of a graph $G$ is the smallest $k$ such that $G$ is a subgraph of a simple $k$-tree. It is not hard to see that we have $$\mathrm{tw}(G)\le \mathrm{stw}(G)\le \mathrm{tw}(G)+1.$$

The main ingredient for proving the bound for graphs with bounded simple treewidth is the following lemma for which we need the well-known fact (see~\cite{JoretMicek2022}, for example) that for every $n$-vertex path $P_n$ we have 
\begin{equation}\label{wcolinfty}
\wcol_\infty(P_n) =\lceil\log_{2}(n+1) \rceil.
\end{equation}

\begin{lemma}\label{gcolpath}
Let $k,\ell$ be positive integers with $k\le \ell$. For every path  $P$ we have $\gcol_{k,\ell}(P)\leq \lceil \log k \rceil +2\lfloor\ell/k\rfloor$.
\end{lemma}
\begin{proof}
The proof builds on that of Theorem~\cite[Theorem 1]{JoretMicek2022}. We enumerate the vertices of $P$ as $v_1,v_2, \dots, v_n$ by going from one end of $P$ to the other. If $k=1$ we take an ordering $L$ that follows this enumeration and notice that for every $v\in V(P)$ we have $\GReach_{k,\ell}[G,L,v]\le 2 \le2\lfloor\ell/k\rfloor.$ So we can assume we have $k\ge 2$, and we let $V_0=\{v_i\in V(P) \mid i=0 \textnormal{ (mod } k)\}$. We define an ordering $L$ on $V(P)$ such that $x<_Ly$ whenever $x\in V_0$ and $y\notin V_0$, and such that for every component $P'$ of $P-V_0$ we have $\wcol_\infty(P',L)\le\lceil \log(k)\rceil$ (we can do this because by \eqref{wcolinfty} each such component has at most $k-1$ vertices).

Note that for any $v\in V(P)$  the number of vertices of $V_0\setminus\{v\}$ that are at distance at most~$\ell$ from $v$ is at most $2\lceil\ell/k \rceil$. If $v\in V_0$, then by construction we have $\GReach_{k,\ell}[G,L,v]\subseteq V_0$, and so $|\GReach_{k,\ell}[G,L,v]|\le 2\lceil\ell/k \rceil+1$. Otherwise, if $v\notin V_0$ the construction gives us $\GReach_{k,\ell}[G,L,v]\subseteq V_0\cup P_{v}$, where $P_v$ is the component of $P-V_0$ containing $v$. If $L_v$ is the restriction of $L$ to $P_v$ then we have $|\GReach_{k,\ell}[P_v,L_v,v]|\le \wcol_{\infty}(P_v,L_v)\le \lceil \log(k)\rceil$, and we obtain $|\GReach_{k,\ell}[G,L,v]|\le  \lceil \log(k)\rceil+2\lceil\ell/k \rceil$, as desired.
\end{proof}

The bounds for graphs with bounded simple treewidth follow from this lemma, using straightforward modification of the layering arguments used in the proof of Theorem 2 in~\cite{JoretMicek2022}




\subsection{Bounded genus and excluded minors}\label{sec:manybounds}

In this section we obtain bounds for the generalised colouring numbers in graphs with certain excluded minors. Our bounds generalise those known for the weak colouring numbers \cite{vandenHeuveletal2017, vandenHeuvelWood2018}, including small improvements mentioned in \cite{KYY2020}. We need some definitions and lemmas which will allow us to use know decompositions to obtain our bounds.

Let $\mathcal{H}=\{H_1,...,H_s\}$ be a decompostion of a graph $G$ (as defined in the previous section). We say $\mathcal{H}$ is \emph{connected} if every $H_i$ is connected. Let $C$ be a component  of $G[H_{\geq{i+1}}]$ with $i\in\{1,...,s-1\}$. The $i$-separating number of $C$, $s_{i}(C)$, is the number $s$ of graphs in $\{H_1,...,H_i\}$ such that they are connected to $C$. Let $w_i(\mathcal{H})$ be the maximum $s_{i}(C)$ over all components $C$ of $G[H_{\geq i}]$. The width of $\mathcal{H}$ is defined as $\max_{1\leq i\leq s-1}w_i(\mathcal{H})$.

A spanning tree $T$ of $G$ rooted at a vertex $r$ is a \emph{BFS spanning tree} if $d_G(v,r)=d_T(v,r)$ for every vertex in $v$.  A \emph{BFS subtree} is a subtree of a BFS spanning tree that includes the root. A \emph{leaf} in a rooted tree is a non-root vertex of degree 1. We will be interested in decompositions where each $H_i$ is induced by a BFS subtree of $G[H_{\geq{i}}]$ with a bounded number of leaves. Since every such subtree is the union of a bounded number of isometric paths, Lemma \ref{naivebound} allows us to bound the number of vertices of $H_i$ that can be reached from some other fixed vertex.




The following theorem is proved in \cite{vandenHeuveletal2017}.

\begin{lemma}[Van den Heuvel et al. \cite{vandenHeuveletal2017}]\label{lemma:decompositionContraction}
Let $G$ be a graph and let $\mathcal{H}=\{H_1,...,H_s\}$ be a connected decomposition of $G$ of width at most $t$. Let $H$ be the graph obtained by contracting each subgraph $H_i$ to a single vertex. Then $H$ has treewidth at most $t$.
\end{lemma}

Using this and our bounds for graphs with bounded treewidth, we now prove a lemma that will allow us to use known decompositions of graphs with excluded minors to obtain bounds on the generalised colouring numbers of these graphs.

\begin{lemma}\label{gcolBFS}
Let $k,\ell,p,t$ be positive integers with $k\le \ell$. Let $G$ be a graph that admits a connected decomposition $\mathcal{H}=\{H_1,...,H_s\}$ of width $t$ in which for every $1\le i \le s$ $H_i$ is induced by a BFS subtree with at most $p$ leaves in $G-(V(H_1)\cup\dots  \cup  V(H_{i-1}))$. Then  we have $$\gcol_{k,\ell}(G)\leq p\large\left(\binom{t+k}{t}-1\large\right)(2\ell +1) +p\ell+1$$.
\end{lemma}
\begin{proof}
The proof is similar to that of \cite[Lemma 3.5]{vandenHeuveletal2017}.
Let $H$ be the graph obtained by contracting the subgraphs $H_i$ in $G$.  We identify the vertices of $H$ with the subgraphs $H_i$. Since $\mathcal{H}$ is connected we have by Lemma \ref{lemma:decompositionContraction} that $H$ has treewidth at most $t$. By our bounds for graphs with bounded treewidth we have $\gcol_{k,\ell}(H)\leq\binom{t+k}{t}$, so there exist a linear ordering $L$ on $V(H)$ such that for every $H_i\in V(H)$ $|\GReach_{k,\ell}[H,L,H_i]|\leq\binom{t+k}{t}$.

From $L$ we define an ordering $L'$ on $V(G)$. For $u\in H_i$ and $v\in H_j$, with $i\neq j$, we let $u<_{L'}v$ if $H_{i}<_{L}H_{j}$. And for every $1\leq i\leq s$ we order the vertices of $H_i$ in such a way that $u<_{L'} v$ if $d_{H_i}(r_i,u)>d_{H_i}(r_i,v)$, where $r_i$ is the root of the BFS subtree that induces $H_i$.

Note that every vertex $v\in H_i$ satisfies
$$\GReach_{k,\ell}[G,L',v]\subseteq N^{\ell}[v]\cap\{H_j|H_j\in\GReach_{k,\ell}[H,L,H_i]\}$$

Hence, we have that there are at most $\binom{t+k}{t}$ subgraphs among $H_1,...,H_s$ in $G$ that contain vertices from $\GReach_{k,\ell}[G,L',v]$. Since each such subgraph $H_i$ is the union of at most $p$ isometric paths, by Lemma~\ref{naivebound}, we get that $|N^{\ell}[v]\cap V(H_i)|\le p(2\ell+1)$. Moreover, if $H_i$ contains $v$ then it is not hard to see that, by construction of $L'$, $|\GReach_{k,\ell}[G,L',v]\cap H_i|\le p\ell+1$. The result follows.
\end{proof}



Now we are ready to reap the remaining results of this section. We start with graphs with bounded genus. The proof of this bound is similar to that of \cite[Theorem 1.6]{vandenHeuveletal2017}.
Since the isometric path decompositions guaranteed by Lemma~\ref{triangulated} are of width 2, and since the generalised colouring numbers cannot decrease if we add edges, Lemma~\ref{gcolBFS} gives us the bound when the genus is $g=0$. (Note that every subgraph of an isometric-path decomposition has two leaves, so we would be first inclined to use $p=2$ in Lemma~\ref{gcolBFS}. But in the proof of Lemma~\ref{gcolBFS} the relevant thing is that every subgraph of the decomposition is the union of at most $p$ isometric paths. Thus for isometric paths decompositions we can use $p=1$.) Now suppose $G$ is a graph with genus $g\ge 1$. Such a graph contains an non-separating cycle $C$ that consists of two isometric paths and such that $G-C$ has genus at most $g-1$  \cite[page 111]{MT01} . We take the vertices of one such a cycle and start constructing a linear $L$ ordering of $V(G)$ by placing these vertices first. If after removing this cycle, the graph obtained has positive genus, we take a cycle of this type, remove it and put its vertices next in the ordering. We proceed like this inductively until we arrive at a planar graph $G'$. We then order the vertices of $G'$ in a way that it satisfies the bound  for $g=0$. The bound on $\gcol_{k,\ell}(G)$ now follows easily from Lemma \ref{naivebound}.

The following is proved in \cite{vandenHeuveletal2017} and together with Lemma \ref{gcolBFS} directly implies the bounds for graph excluding $K_t$ as a minor.

\begin{lemma}[Van den Heuvel et al. \cite{vandenHeuveletal2017}]\label{Kminor}
Every $K_t$-minor free graph $G$ has a connected partition $H_1,\dots, H_s$ with width at most $t-2$,  where each $H_i$ is induced by a BFS subtree of $G-(V(H_1)\cup \dots \cup V(H_{i-1}))$ with at most $t-3$ leaves.
\end{lemma}

The following lemmas are proved in \cite{vandenHeuvelWood2018} and together with Lemma \ref{gcolBFS} (or, in the last case, the arguments from the proof of this lemma) imply the bounds for graphs excluding $K^*_{2,t}$, $K^*_{3,t}$ or $K^*_{s,t}$ as a minor, respectively.

\begin{lemma}[Van den Heuvel and Wood \cite{vandenHeuvelWood2018}]
Every $K^*_{2,t}$-minor free graph $G$ has a connected partition $H_1,\dots, H_s$ with width $1$,  where each $H_i$ is induced by a BFS subtree of $G-(V(H_1)\cup \dots \cup V(H_{i-1}))$ with at most $t-1$ leaves.
\end{lemma}

\begin{lemma}[Van den Heuvel and Wood \cite{vandenHeuvelWood2018}]
Every $K^*_{3,t}$-minor free graph $G$ has a connected partition $H_1,\dots, H_s$ with width $2$,  where each $H_i$ is induced by a BFS subtree of $G-(V(H_1)\cup \dots \cup V(H_{i-1}))$ with at most $2t+1$ leaves.
\end{lemma}

\begin{lemma}[Van den Heuvel and Wood \cite{vandenHeuvelWood2018}]
Every $K^*_{s,t}$-minor free graph $G$ has a connected partition $H_1,\dots, H_s$ with width $s$ in which for every $1\le i \le r$ $V(H_i)=V(P_{i,1}\cup\dots\cup V(P_{i,p_i}))$, where $p_i\le s(t-1)$ and each $P_{i,j}$ is an isometric path in $G-((V(H_1)\cup\dots V(H_{i-1}))\cup(V(P_{i,1})\cup\dots\cup V(P_{i,j-i})))$.
\end{lemma}

\section{Approximation algorithms for the subchromatic number of powers of planar graphs}\label{sec:approx}

In this section we give a $2$-approximation for computing the subchromatic number of graph classes with algorithmically bounded layered cliquewidth, which in particular will give a $2$-approximation for computing the subchromatic number of powers of planar graphs, even if we are not given the underlying planar graph. For clarity, we first include a proof that powers of planar graphs have algorithmically bounded layered cliquewidth. Of course, this works in more general graph classes as well.

\begin{lemma}\label{lem:clarify}
Let $d$ be a fixed positive integer, and let $G$ be a connected planar graph. Then there exists $c=c(d)$ such that for any spanning tree $T$ of $G^{d}$, where $(L_{1},\ldots,L_{t})$ is the associated layering with $L_{1} = \{r\}$, we have that $G^{d}[L_{i}]$ has cliquewidth at most $c$. In particular, this layering is computable in polynomial time.
\end{lemma}

\begin{proof}
Let $T'$ be a spanning tree of $G$ and let $(L_{1}',\ldots,L_{q}')$ be the associated layering where $L_{1}' = \{r\}$. By combining Lemma 6 and Theorem 11 of \cite{productstructuretheorem}, there exists a tree decomposition $(T',\beta)$ of $G$ for which each bag intersects $L'_{i}$ in at most $9$ vertices. Thus for any set of $d$ layers, any bag intersects these layers in at most $9d$ vertices. Thus if $H$ is an induced subgraph of any $d$ consectutive layers has treewidth at most $9d$. Observe that in the layering of $G^{d}$, a layer $L_{i}$ is the $d^{th}$ power of at most $d$ consecutive layers of $(L_{1}',\ldots,L_{q}')$. Thus $L_{i}$ is the power of a graph with  treewidth at most $9d$, and thus has  cliquewidth at most $c(d):= 2(d+1)^{9d+1}$ \cite{GURSKI2009583} . As $L_{i}$ is an arbitrarily layer, this implies that every layer of $G^{d}$ has  cliquewidth at most $c(d)$. 
\end{proof}



\begin{proof}[Proof of Theorem \ref{thm:approximationalgo}]
Let $G \in \mathcal{C}$, and without loss of generality we will assume that $G$ is connected, as otherwise we simply apply the algorithm to each connected component.  Let $(L_{1},\ldots,L_{t})$ be a layering of $G$ such that each layer has bounded cliquewidth, which exists as $G$ has bounded layered cliquewidth, and further by the assumption, can be computed in polynomial time.
Observe that we can check if a graph $H$ has subchromatic number at most $t$ by checking the following $\text{MSO}_{1}$-formula:
\[
\begin{split}
\exists V_1,\dots,V_t\quad &\Bigl(\forall v\ \bigvee_{i=1}^t v\in V_k\Bigr)\\
&\wedge\bigwedge_{i=1}^t\Bigl(\forall u,v,w\ \bigl(u\in V_i\wedge v\in V_i\wedge w\in V_i\wedge E(u,v)\wedge E(v,w)\bigr)\rightarrow E(u,w)\Bigr).
\end{split}
\]
By \cite{COURCELLE20031}, this formula can be checked in polynomial time on graphs of bounded cliquewidth. Therefore, we can determine exactly the subchromatic number of the graph $L_{i}$ for all $i \in \{1,\ldots,k\}$. Let $t$ be the maximum subchromatic number in the layers. Then by using a set of $t$ colours on odd layers, and a set of $t$ different colours on even layers, as there is no edge between two layers of the same parity, this gives a $2t$-subcolouring of $G$, and hence a $2$-approximation for the subchromatic number of $G$. 
\end{proof}

We note that this algorithm can in fact find a subcolouring that uses at most twice the optimal number of parts. Indeed, we can put $t$ marks (i.e. unary predicates) $A_1,\dots,A_t$ on the vertices of $G$ and test whether a layer admits a subcolouring with $t$ colours extending $A_1,\dots,A_t$ using the algorithm in \cite{COURCELLE20031}. Doing this on each layer means we can find a subcolouring using $t$-colours, and thus find a $2t$-subcolouring of $G$.

 For our algorithm we used that powers of planar graphs have algorithmically bounded layered cliquewidth. 
We note that this algorithm works more generally than just for the subchromatic number as all we needed was that layers of the BFS tree had bounded cliquewidth, and that we could compute exactly the subchromatic number on graphs of bounded cliquewidth. For example, consider the \emph{$c$-chromatic number} of a graph, which is the minimum size of a vertex partition with the property that every class induces a cograph \cite{Gimbel20103437}. It is known that every first-order transduction of a class with bounded expansion has bounded $c$-chromatic number \cite{SBETOCL} (and refer to \cite{SBETOCL} for definitions of first-order transductions). The exact same argument in Theorem \ref{thm:approximationalgo} shows that there exists a $2$-approximation to the $c$-chromatic number assuming our class is a strongly local transduction of a bounded expansion class with bounded layered treewidth and such that the bounded cliquewidth layering is computable in polynomial time (see \cite{NesetrilMS22} for formal definition of strongly local transduction). Even more generally, the exact same algorithm can compute a $(p+1)$-low-shrubdepth  cover (see \cite{SBETOCL} for a definition) of such graph classes with parameter $p$ in polynomial time, which gives a partial answer to the question in \cite{SBETOCL}. We note that
by a small extension of the result on neighborhood covers presented in \cite{arxiv.2302.03527}, we can compute in 
$O(n^{9.8})$-time an $O((\log n)^2)$-subcolouring of $G^d$, when $G$ is a graph of order $n$ in a fixed class with bounded expansion, or an $O(n^\epsilon)$-subcolouring of $G^d$ for $G$ of order $n$ in a fixed nowhere dense class.

\section{Open problems}\label{sec:fur}


The bounds obtained in \cite{vandenHeuveletal2017} for the $\wcol_\ell(G)$ when $G$ is planar or excludes a fixed minor are polynomial in $\ell$ (see Table \ref{tab:weak}), and thus imply polynomial bounds on $\col_\ell(G)$ for these graphs. Joret and Wood (see \cite{esperetraymond2018}), asked if every graph class with polynomial upper bounds on the strong colouring numbers also has polynomial bounds on the weak colouring numbers. This turns out not to be the case as shown by Grohe, Kreutzer, Rabinovich, Siebertz and Stavropoulos \cite{grohe2018coloring} and by Dvo\v{r}\'ak, Pek\'arek, Ueckerdt, and Yuditsky \cite{DPUY}, where the second paper gives a more natural class of graphs not satisfying the property. We ask then, the following question. 

\begin{question}
 
 What is the largest $k=k(\ell)$ such that having polynomial bounds on the strong colouring numbers guarantees polynomial bounds on $\gcol_{k,\ell}$?
 \end{question}


It would be interesting to the search for improved lower bounds for $\chisub(G^d)$ when $G$ belongs to a minor closed class. The most immediate question (for which Theorem~\ref{thm:gcol_bounds} might be of help) is the following.

\begin{question}
    What is the maximum value of $\chisub(G^d)$ when $G$ has treewidth at most $t$?
\end{question}

It would also be interesting to find improved lower bounds for the subchromatic number of squares of planar graphs. Currently the best known lower bound is five \cite{clusteringpowers}, leaving a large gap from our upper bound of $43$.

\begin{ack}
We thank an anonymous referee for pointing out an error in an earlier version of  Theorem \ref{boundingsemiweak}.
\end{ack}

\bibliographystyle{plain}
\bibliography{bibliography}
\end{document}